\documentclass{article}
\usepackage{amssymb}
\usepackage{amsfonts}
\usepackage{amsmath}

\newtheorem{theorem}{Theorem}

\newtheorem{lemma}[theorem]{Lemma}

\newtheorem{proposition}[theorem]{Proposition}
\newtheorem{remark}[theorem]{Remark}

\newenvironment{proof}[1][Proof]{\noindent\textbf{#1.} }{\ \rule{0.5em}{0.5em}}
\input{tcilatex}
\begin{document}

\title{\textbf{Equivariant holonomy of U(1)-bundles}}
\date{}
\author{\textsc{Roberto Ferreiro P\'{e}rez} \\
Departamento de Econom\'{\i}a Financiera y Actuarial y Estad\'{\i}stica\\
Facultad de Ciencias Econ\'omicas y Empresariales, UCM\\
Campus de Somosaguas, 28223-Pozuelo de Alarc\'on, Spain\\
\emph{E-mail:} \texttt{roferreiro@ccee.ucm.es}}
\maketitle

\begin{abstract}
We define the equivariant holonomy of an invariant connection on a principal 
$U(1)$-bundle. The properties of the ordinary holonomy are generalized to
the equivariant setting. In particular, equivariant $U(1)$-bundles with
connection are shown to be classified by its equivariant holonomy modulo
isomorphisms. We also show that the equivariant holonomy can be used to
obtain results about equivariant prequantization and anomaly cancellation.
\end{abstract}


\smallskip


\medskip

\noindent \emph{Acknowledgments:\/} Supported by Ministerio de Ciencia,
Innovaci\'{o}n y Universidades of Spain under grant PGC2018-098321-B-I00.

\section{Introduction}

When a group $G$ acts on a principal bundle $P\rightarrow M$, we can
consider the $G$-equivariant version of almost any mathematical
construction. For example we have $G$-equivariant\ homology and cohomology (%
\cite{GS}), $G$-equivariant characteristic classes (\cite{BV2}), $G$%
-equivariant differential cohomology (\cite{Gomi},\cite{kubel}), etc.
However, in our study of anomaly cancellation in \cite{AnomaliesG} we needed
to consider the $G$-equivariant holonomy of a $G$-invariant\ connection on a 
$U(1)$-bundle. Although there is a natural definition of this concept, we
have been unable\ to find a detailed study of it\ in the literature (and we
are not the only ones to face this problem, e.g. see \cite[page 8]{EquiChern}%
). In \cite{AnomaliesG} the $G$-equivariant holonomy is studied in the
particular case of topologically trivial $U(1)$-bundles over a contractible
space. The reason is that this is the case that appears in the study of
gravitational anomaly cancellation (see Section \ref{anomalies} for more
details).

In the present paper we study the $G$-equivariant holonomy of a $G$%
-invariant\ connection $\Xi$ on an arbitrary $G$-equivariant $U(1)$-bundle $%
p\colon P\rightarrow M$. We show that the basic properties of ordinary
holonomy can be extended to $G$-equivariant holonomy. In particular, we
prove that the $G$-equivariant holonomy\ classifies $G$-equivariant $U(1)$%
-bundles with connection modulo isomorphisms. We give an application of the
equivariant holonomy to study the existence of equivariant prequantization
bundles. Finally we study the relation of equivariant holonomy with the
study of anomaly cancellation.

To motivate our definition of equivariant holonomy, let us consider the case
of a free and proper action of a discrete group $G$ on a $U(1)$-bundle $%
p\colon P\rightarrow M$. If $\Xi$ is a $G$-invariant connection on $%
P\rightarrow M$, then it projects onto a connection $\underline{\Xi}$ on the
quotient bundle $P/G\rightarrow M/G$. We want to compute the holonomy of $%
\underline{\Xi}$ in terms of $\Xi$. A loop $\underline{\gamma}$ on $M/G$
based at $[x]\in M/G$ can be represented by a curve $\gamma\colon
\lbrack0,1]\rightarrow$ $M$ such that $\gamma(0)=x$ and $\gamma(1)=\phi
_{M}(x)$ for an element $\phi\in G$. If $y\in p^{-1}(x)\subset P$, we can
consider the $\Xi$-horizontal lift $\overline{\gamma}^{y}\colon I\rightarrow
P$ with $\overline{\gamma}^{y}(0)=y$. The problem is that $\overline{\gamma }%
^{y}(0)$ and $\overline{\gamma}^{y}(1)$ belong to different fibers, and
hence we cannot compare them in order to define the ordinary holonomy.
However, $\overline{\gamma}^{y}(1)$ and $\phi_{P}(y)$ are in the same fiber,
and hence they can be related. The $\phi$-equivariant holonomy $\mathrm{Hol}%
_{\phi}^{\Xi}(\gamma)\in U(1)$\ of $\gamma$ is characterized by the property 
$\overline{\gamma}^{y}(1)=\phi_{P}(y)\cdot\mathrm{Hol}_{\phi}^{\Xi}(\gamma)$%
. Moreover, we prove that we have $\mathrm{Hol}_{\phi}^{\Xi}(\gamma )=%
\mathrm{Hol}^{\underline{\Xi}}(\underline{\gamma})$.

For the action of an arbitrary Lie group $G$ on $P\rightarrow M$, we define
the equivariant holonomy\ in a similar way. As it is well known, the
curvature of a connection $\mathrm{curv}(\Xi)$ measures the infinitesimal
holonomy of $\Xi$. In the equivariant case, we have the equivariant
curvature $\mathrm{curv}_{G}(\Xi)=\mathrm{curv}(\Xi)+\mu^{\Xi}$ (e.g. see 
\cite{BV2}), where $\mu^{\Xi}\colon\mathfrak{g}\rightarrow\mathbb{R}$ is
called the momentum of $\Xi$. We show that $\mu^{\Xi}$ measures the
infinitesimal variation of the equivariant holonomy $\mathrm{Hol}%
_{\phi}^{\Xi}(\gamma)$ with respect to $\phi\in G$.

\bigskip

The results of this paper can be generalized in several ways. For example,
the equivariant holonomy can be defined for connections on principal bundles
with arbitrary structural group. We left the study of the general case for a
separate paper.

Another possible extension is the following. In \cite{Cheeger} the concept
of differential character is introduced as an object similar to the holonomy
of a connection. In the equivariant case, we can define an equivariant
differential character (of second order) as a map $\chi $ which assigns a
complex number $\chi (\phi ,\gamma )\in U(1)$ to a pair $(\phi ,\gamma )$,
with $\phi \in G$ and $\gamma \colon \lbrack 0,1]\rightarrow $ $M$ such that 
$\gamma (1)=\phi _{M}(\gamma (0))$. This concept of equivariant differential
character has been introduced in \cite{LermanMalkin}\ in order to classify
equivariant $U(1)$-bundles modulo isomorphisms. Furthermore, it can be seen
that there are important geometrical objects that appear in a natural way as
equivariant differential characters not necessarily associated to an
invariant connection. One example is the definition of the Chern-Simons line
bundle (see \cite{CSconnections}) and another example is Witten's formula
for global anomalies (see Section \ref{anomalies}).

\section{Holonomy on $U(1)$-bundles\label{holonomy}}

In this section we recall some classical results about the holonomy of
connections on $U(1)$-bundles (e.g. see \cite{Brylinsky}, \cite{Kostant} for
more details). In the rest of the paper we show how to extend these results
to the equivariant setting.

The interval $[0,1]$ is denoted by $I$. The space of loops on $M$ is defined
by $\mathcal{C}(M)=\{\gamma \colon I\rightarrow M$ $|$ $\gamma $ is
piecewise $\mathcal{C}^{1}$ and $\gamma (1)=\gamma (0)\}$, and the space of
loops based at $x$ is defined by $\mathcal{C}_{x}(M)=\{\gamma \in \mathcal{C}%
(M)$ $|$ $\gamma (0)=x\}$. The holonomy can be defined for connections on
principal bundles with arbitrary structural group (e.g. see \cite{KN1}). Let 
$K$ be a Lie group, $p\colon P\rightarrow M$ a principal $K$-bundle, and let 
$\Xi $ be a connection on $P\rightarrow M$. If $\gamma \in \mathcal{C}%
_{x}(M) $ and $y\in p^{-1}(x)$, we have a $\Xi $-horizontal lift $\overline{%
\gamma }^{y}\colon I\rightarrow P$ with $\overline{\gamma }^{y}(0)=y$.
Furthermore, if $\gamma \in \mathcal{C}_{x}(M)$, then we have $p(\overline{%
\gamma }^{y}(1))=p(\overline{\gamma }^{y}(0))=x$ and hence there exist 
\textrm{Hol}$^{\Xi ,y}(\gamma )\in K$ such that $\overline{\gamma }^{y}(1)=%
\overline{\gamma }^{y}(0)\cdot $\textrm{Hol}$^{\Xi ,y}(\gamma )$. If $K$ is
abelian then it can be seen that \textrm{Hol}$^{\Xi ,y}(\gamma )$ is
independent of the element $y\in \pi ^{-1}(x)$ chosen and it is denoted
simply by \textrm{Hol}$^{\Xi }(\gamma )$.

Two curves $\gamma _{1}$ and $\gamma _{2}$ are said to differ by a
reparametrization if there exists an orientation preserving homeomorphism $%
\varphi \colon I\rightarrow I$ such that $\varphi $ and $\varphi ^{-1}$ are
piecewise $\mathcal{C}^{1}$ and $\gamma _{2}=\gamma _{1}\circ \varphi $. In
this case we have $\overline{\gamma _{2}}^{y}(1)=\overline{\gamma }%
_{1}^{y}(1)$ (e.g. see \cite{KN1}). Hence if $\gamma _{1},\gamma _{2}\in 
\mathcal{C}_{x}(M)$ differ by a reparametrization then \textrm{Hol}$^{\Xi
}(\gamma _{1})=$\textrm{Hol}$^{\Xi }(\gamma _{2}).$

Now we consider the case $K=U(1)$. If $\gamma\colon I\rightarrow M$ is a
curve, we define the inverse curve $\overleftarrow{\gamma}\colon
I\rightarrow M$ by $\overleftarrow{\gamma}(t)=\gamma(1-t)$. If $\gamma\in%
\mathcal{C}_{x}(M)$ then we have \textrm{Hol}$^{\Xi}(\overleftarrow{\gamma})=%
\mathrm{Hol}^{\Xi}(\gamma)^{-1}$. Moreover, if $\gamma_{1}$ and $\gamma _{2}%
\mathcal{\ }$are curves with $\gamma_{1}(1)=\gamma_{2}(0)$ we define $%
\gamma_{1}\ast\gamma_{2}$ by $\gamma_{1}\ast\gamma_{2}\colon I\rightarrow 
\mathbb{R}$, $\gamma_{1}\ast\gamma_{2}(t)=\gamma_{1}(2t)$ for $t\in
\lbrack0,1/2]$ and $\gamma_{1}\ast\gamma_{2}(t)=\gamma_{2}(2t-1)$ for $%
t\in\lbrack1/2,1]$. If $\gamma_{1},\gamma_{2}\in\mathcal{C}_{x}(M)$ then we
have $\gamma_{1}\ast\gamma_{2}\in\mathcal{C}_{x}(M)$ and \textrm{Hol}$^{\Xi
}(\gamma_{1}\ast\gamma_{2})=$\textrm{Hol}$^{\Xi}(\gamma_{1})\cdot \mathrm{Hol%
}^{\Xi}(\gamma_{2})$.

The connection is a form $\Xi\in\Omega^{1}(P,i\mathbb{R})$ and the curvature
form$\ \mathrm{curv}(\Xi)\in\Omega^{2}(M)$ is defined by the property $%
p^{\ast}(\mathrm{curv}(\Xi))=\frac{i}{2\pi}d\Xi$. As it is well known, for
bundles with arbitrary group the curvature of $\Xi$\ measures the
infinitesimal holonomy. For $U(1)$-bundles we have a more precise result
that is a generalization of the classical Gauss-Bonnet Theorem

\begin{proposition}
\label{HolonomyInt}If $\Sigma\subset M$ is a $2$-dimensional submanifold
with boundary $\partial\Sigma=\bigcup\limits_{i=1}^{k}\gamma_{i}$, with $%
\gamma _{i}\in\mathcal{C}(M)$ then\ we have $\prod\limits_{i=1}^{k}\mathrm{%
Hol}^{\Xi }(\gamma_{i})=\exp(2\pi i\int_{\Sigma}\mathrm{curv}(\Xi))$.
\end{proposition}

We also have the following

\begin{proposition}
\label{PropChangeConn copy(1)}If $\Xi$ and $\Xi^{\prime}$ are connections on 
$P\rightarrow M$, then we have $\Xi^{\prime}=\Xi-2\pi i(p^{\ast}\rho)$ for a
form $\rho\in\Omega^{1}(M)$ and $\mathrm{Hol}^{\Xi^{\prime}}(\gamma )=%
\mathrm{Hol}^{\Xi}(\gamma)\cdot\exp(2\pi i\int_{\gamma}\rho)$ for any $%
\gamma\in\mathcal{C}(M)$.
\end{proposition}

If $p\colon P\rightarrow M$, $p^{\prime }\colon P^{\prime }\rightarrow M$
are two principal $U(1)$-bundles, we write that $P\simeq P^{\prime }$\ if
there exists a $U(1)$-bundle isomorphism $\Phi \colon P^{\prime }\rightarrow
P$ covering the identity map of $M$. And $P\rightarrow M$ is a trivial $U(1)$%
-bundle if $P\simeq M\times U(1)$. A $U(1)$-bundle with connection is a pair 
$(P,\Xi )$, where $p\colon P\rightarrow M$ is a $U(1)$-bundle and $\Xi $ is
a connection on $P$. We write that $(P,\Xi )\simeq (P^{\prime },\Xi ^{\prime
})$ if there exists a $U(1)$-bundle isomorphism $\Phi \colon P^{\prime
}\rightarrow P$ covering the identity map of $M$ such that\ $\Phi ^{\ast
}\Xi =\Xi ^{\prime }$.

The holonomy can be used to classify $U(1)$-bundles modulo isomorphisms.
Precisely, we recall the following classical result (e.g. see \cite[Theorem
2.5.1]{Kostant})

\begin{theorem}
\label{Clas} If $(P,\Xi)$ and $(P^{\prime},\Xi^{\prime})$ are $U(1)$-bundles
with connection over $M$, then $(P,\Xi)\simeq(P^{\prime},\Xi^{\prime})$ if
and only if $\mathrm{Hol}^{\Xi}(\gamma)=\mathrm{Hol}^{\Xi^{\prime}}(\gamma)$
for any $\gamma\in\mathcal{C}(M)$.
\end{theorem}

A consequence of the preceding theorem and Proposition \ref{PropChangeConn
copy(1)} is the following

\begin{proposition}
Let $(P,\Xi)$ be a $U(1)$-bundle with connection over $M$. Then $%
P\rightarrow M$ is a trivial $U(1)$-bundle if and only if there exists a $1$%
-form $\beta \in\Omega^{1}(M)$ such that $\mathrm{Hol}^{\Xi}(\gamma)=\exp(2%
\pi i\int_{\gamma}\beta)$ for any $\gamma\in\mathcal{C}(M)$.
\end{proposition}

A connection $\Xi$ is flat if $\mathrm{curv}(\Xi)=0$. In this case, it
follows from Proposition \ref{HolonomyInt} that $\mathrm{Hol}^{\Xi}(\gamma)$
depends only on the homotopy class of $\gamma$. Hence the holonomy of a flat
connection defines a homomorphism $\mathrm{Hol}^{\Xi}\colon\pi_{1}(M)%
\rightarrow U(1)$.

If $\omega\in\Omega^{2}(M)$ is closed, then $\omega$ is prequantizable if
there exist a $U(1)$-bundle with connection $(P,\Xi)$ such that $\mathrm{curv%
}(\Xi)=\omega$. By a classical result of Weil and Kostant (e.g. see \cite[%
Proposition 2.1.1]{Kostant}), $\omega$ is prequantizable if and only if it
is integral, i.e., if its de Rham cohomology class comes from an integral
class under the natural map $H^{2}(M,\mathbb{Z})\rightarrow H^{2}(M,\mathbb{R%
})$.

\section{Equivariant holonomy}

In this section we define the equivariant holonomy and the notations that
are used in the rest of the paper. Let $G$ be a Lie group with Lie algebra $%
\mathfrak{g}$ and let $M$ be a connected manifold. A $G$-equivariant $U(1)$%
-bundle is a principal $U(1)$-bundle $p\colon P\rightarrow M$ in which $G$
acts (on the left) by principal bundle automorphisms. If $\phi \in G$ and $%
y\in P$, we denote by $\phi _{P}(y)$ the action of $\phi $ on $y$. In a
similar way, for $X\in \mathfrak{g}$ we denote by $X_{P}\in \mathfrak{X}(P)$
the corresponding vector field on $P$ defined by$\footnote{%
In the definition of the fundamental vector field $X_{P}$ we follow the sign
convention of \cite[page 10]{GS}}$ $X_{P}(x)=\left. \frac{d}{dt}\right\vert
_{t=0}\exp (-tX)_{P}(x)$. For any $\phi \in G$ we define%
\begin{equation*}
\mathcal{C}^{\phi }(M)=\{\gamma \colon I\rightarrow M\text{ }|\text{ }\gamma 
\text{ is piecewise }\mathcal{C}^{1}\text{ and }\gamma (1)=\phi _{M}(\gamma
(0))\},
\end{equation*}%
and $\mathcal{C}_{x}^{\phi }(M)=\{\gamma \in \mathcal{C}^{\phi }(M)$ $|$ $%
\gamma (0)=x\}$. Note that if $e\in G$ is the identity element, then $%
\mathcal{C}_{x}^{e}(M)=\mathcal{C}_{x}(M)$ is the space of loops based at $x$%
. If $\phi \in G$ and $\gamma \in \mathcal{C}_{x}^{\phi ^{\prime }}(M)$ then
we define $\phi \cdot \gamma \in \mathcal{C}_{\phi _{M}(x)}^{\phi \cdot \phi
^{\prime }\cdot \phi ^{-1}}(M)$ by $(\phi \cdot \gamma )(t)=\phi _{M}(\gamma
(t))$.

Let $\Xi $ be a $G$-invariant connection on a $G$-equivariant $U(1)$-bundle $%
P\rightarrow M$. If $\gamma \in \mathcal{C}_{x}^{\phi }(M)$ and $y\in
p^{-1}(x)$, we have a $\Xi $-horizontal lift $\overline{\gamma }^{y}\colon
I\rightarrow P$ with $\overline{\gamma }^{y}(0)=y$. We have $p(\overline{%
\gamma }^{y}(1))=p(\phi _{P}(y))=\phi _{M}(x)$, and hence there exists $u\in
U(1)$ such that $\overline{\gamma }^{y}(1)=(\phi _{P}(y))\cdot u$. As $%
\overline{\gamma }^{y\cdot z}=\overline{\gamma }^{y}\cdot z$ for $z\in U(1)$%
, it follows that $u$ does not depend on the $y\in p^{-1}(x)$ chosen and we
denote it by $\mathrm{Hol}_{\phi }^{\Xi }(\gamma )\in U(1)$. Hence the
equivariant holonomy is characterized by the property 
\begin{equation}
\overline{\gamma }^{y}(1)=\phi _{P}(y)\cdot \mathrm{Hol}_{\phi }^{\Xi
}(\gamma ).  \label{def}
\end{equation}%
Note that if $\gamma \in \mathcal{C}_{x}^{e}(M)$ is a loop on $M$, then $%
\mathrm{Hol}_{e}^{\Xi }(\gamma )=\mathrm{Hol}^{\Xi }(\gamma )$ is the
ordinary holonomy of $\gamma $.\ Furthermore, if $\gamma $,$\gamma ^{\prime
}\in \mathcal{C}^{\phi }(M)$ differ by a reparametrization then we have $%
\mathrm{Hol}_{\phi }^{\Xi }(\gamma ^{\prime })=\mathrm{Hol}_{\phi }^{\Xi
}(\gamma )$.

\begin{proposition}
\label{PropHol}If $P\rightarrow M$ is a $G$-equivariant principal $U(1)$%
-bundle, and $\Xi$ is a $G$-invariant connection on $P$, then for any $\phi$%
, $\phi^{\prime}\in G$\ and $x\in M$ we have

a) If $\gamma\in\mathcal{C}^{\phi^{\prime}}(M)$ then $\phi\cdot\gamma \in%
\mathcal{C}^{\phi\cdot\phi^{\prime}\cdot\phi^{-1}}(M)$\ and $\mathrm{Hol}%
_{\phi\cdot\phi^{\prime}\cdot\phi^{-1}}^{\Xi}(\phi\cdot\gamma)=\mathrm{Hol}%
_{\phi^{\prime}}^{\Xi}(\gamma)$.

b) If $\gamma\in\mathcal{C}^{\phi}(M)$ and $\gamma^{\prime}\in\mathcal{C}%
_{\gamma(1)}^{\phi^{\prime}}(M)$, then $\gamma\ast\gamma^{\prime}\in\mathcal{%
C}^{\phi^{\prime}\cdot\phi}(M)$ and we have $\mathrm{Hol}_{\phi^{\prime}%
\cdot\phi}^{\Xi}(\gamma\ast\gamma^{\prime})=\mathrm{Hol}_{\phi
}^{\Xi}(\gamma)\cdot\mathrm{Hol}_{\phi^{\prime}}^{\Xi}(\gamma^{\prime}).$

c) If $\gamma\in\mathcal{C}^{\phi}(M)$ then $\overleftarrow{\gamma}\in%
\mathcal{C}^{\phi^{-1}}(M)$ and $\mathrm{Hol}_{\phi^{-1}}^{\Xi }(%
\overleftarrow{\gamma})=\mathrm{Hol}_{\phi}^{\Xi}(\gamma)^{-1}$.

d) If $\gamma,\gamma^{\prime}\in\mathcal{C}_{x}^{\phi}(M)$ then $\gamma
^{\prime}\ast\overleftarrow{\gamma}\in\mathcal{C}(M)$ and $\mathrm{Hol}^{\Xi
}(\gamma^{\prime}\ast\overleftarrow{\gamma})=\mathrm{Hol}_{\phi}^{\Xi}(%
\gamma^{\prime})\cdot\mathrm{Hol}_{\phi}^{\Xi}(\gamma)^{-1}$.

e) If $\zeta\colon I\rightarrow M$ is a curve on $M$ such that $\zeta
(0)=\gamma(0)$ and $\gamma\in\mathcal{C}^{\phi}(M)$\ then $\overleftarrow {%
\zeta}\ast\gamma\ast(\phi\cdot\zeta)\in\mathcal{C}^{\phi}(M)$ and $\mathrm{%
Hol}_{\phi}^{\Xi}(\overleftarrow{\zeta}\ast\gamma\ast(\phi\cdot \zeta))=%
\mathrm{Hol}_{\phi}^{\Xi}(\gamma)$.

f) Let $P^{\prime}\rightarrow M^{\prime}$ be another $G$-equivariant $U(1)$%
-bundle with connection and $\Phi\colon P^{\prime}\rightarrow P$ be a $G$%
-equivariant $U(1)$-bundle morphism with covers $\underline{\Phi}\colon
M^{\prime}\rightarrow M$. The connection $\Xi^{\prime}=\Phi^{\ast}\Xi$ is $G$%
-invariant and we have $\mathrm{Hol}_{\phi}^{\Xi^{\prime}}(\gamma )=\mathrm{%
Hol}_{\phi}^{\Xi}(\underline{\Phi}\circ\gamma)$ for any $\phi\in G$ and $%
\gamma\in\mathcal{C}^{\phi}(M^{\prime})$.
\end{proposition}

\begin{proof}
a) For any $\gamma\in\mathcal{C}^{\phi^{\prime}}(M)$ and $y\in
p^{-1}(\gamma(0))$ by equation (\ref{def}) we have $\overline{\gamma}%
^{y}(1)=\phi_{P}^{\prime}(y)\cdot\mathrm{Hol}_{\phi^{\prime}}^{\Xi}(\gamma)$
and $\overline{\phi\cdot\gamma}^{\phi_{P}(y)}(1)=(\phi\cdot\phi^{%
\prime})_{P}(y)\cdot\mathrm{Hol}_{\phi\cdot\phi^{\prime}\cdot\phi^{-1}}^{%
\Xi}(\phi\cdot\gamma)$.\ Using that $\Xi$ is $G$-invariant we obtain $\phi
\cdot\overline{\gamma}^{y}=\overline{\phi\cdot\gamma}^{\phi_{P}(y)}$, and
hence 
\begin{align*}
\overline{\phi\cdot\gamma}^{\phi_{P}(y)}(1) & =(\phi\cdot\overline{\gamma }%
^{y})(1)=\phi_{P}(\overline{\gamma}^{y}(1))=\phi_{P}(\phi_{P}^{\prime
}(y)\cdot\mathrm{Hol}_{\phi^{\prime}}^{\Xi}(\gamma)) \\
& =\phi_{P}(\phi_{P}^{\prime}(y))\cdot\mathrm{Hol}_{\phi^{\prime}}^{\Xi
}(\gamma)=(\phi\cdot\phi^{\prime})_{P}(y)\cdot\mathrm{Hol}%
_{\phi^{\prime}}^{\Xi}(\gamma).
\end{align*}
We conclude that we have $\mathrm{Hol}_{\phi\cdot\phi^{\prime}\cdot%
\phi^{-1}}^{\Xi}(\phi\cdot\gamma)=\mathrm{Hol}_{\phi^{\prime}}^{\Xi}(\gamma)$%
.

b) If $\gamma\in\mathcal{C}^{\phi}(M)$ and $y\in\pi^{-1}(\gamma(0))$ then we
have%
\begin{align*}
\overline{\gamma\ast\gamma^{\prime}}^{y}(1) & =\overline{\gamma^{\prime}}^{%
\overline{\gamma}^{y}(1)}(1)=\phi_{P}^{\prime}(\overline{\gamma}^{y}(1))\cdot%
\mathrm{Hol}_{\phi^{\prime}}^{\Xi}(\gamma^{\prime})=\phi
_{P}^{\prime}(\phi_{P}(y)\cdot\mathrm{Hol}_{\phi}^{\Xi}(\gamma))\cdot 
\mathrm{Hol}_{\phi^{\prime}}^{\Xi}(\gamma^{\prime}) \\
& =\phi_{P}^{\prime}(\phi_{P}(y))\cdot\mathrm{Hol}_{\phi}^{\Xi}(\gamma )\cdot%
\mathrm{Hol}_{\phi^{\prime}}^{\Xi}(\gamma^{\prime})=(\phi^{\prime}\cdot%
\phi)_{P}(y))\cdot\mathrm{Hol}_{\phi}^{\Xi}(\gamma)\cdot\mathrm{Hol}%
_{\phi^{\prime}}^{\Xi}(\gamma^{\prime})\text{.}
\end{align*}

c) follows from the fact that $\overleftarrow{\overline{\gamma}^{y}}$ is a
horizontal lift of $\overleftarrow{\gamma}$ and property d) follows from b)
and c).

e) If $y\in p^{-1}(\zeta(0))$ we define $y^{\prime}=\overline{\zeta}^{y}(1)$
and we have 
\begin{align*}
\overline{\overleftarrow{\zeta}\ast\gamma\ast(\phi\cdot\zeta)}%
^{y^{\prime}}(1) & =\overline{(\phi\cdot\zeta)}^{\overline{\gamma}%
^{y}(1)}(1)=\overline{(\phi\cdot\zeta)}^{\phi_{P}(y)}(1)\cdot\mathrm{Hol}%
_{\phi}^{\Xi }(\gamma) \\
& =(\phi\cdot\overline{\zeta}^{y})(1)\cdot\mathrm{Hol}_{\phi}^{\Xi}(\gamma)=%
\phi_{P}(y^{\prime})\cdot\mathrm{Hol}_{\phi}^{\Xi}(\gamma).
\end{align*}

f) follows easily from the properties of parallel transport.
\end{proof}

We recall that the equivariant holonomy $\mathrm{Hol}_{\phi}^{\Xi}(\gamma)$
depends on the curve $\gamma$, but it also depends on $\phi\in G$. For
example, if $x\in M$, $\gamma_{x}$ is the constant curve with value $x$ and $%
G_{x}$ is the isotropy group of $x$, then by Proposition \ref{PropHol} b) we
have a homomorphism $\chi_{x}^{\Xi}\colon G_{x}\rightarrow U(1)$ defined by $%
\chi_{x}^{\Xi}(\phi)=\mathrm{Hol}_{\phi}^{\Xi}(\gamma_{x})$.

The following example shows that the equivariant holonomy can be computed
geometrically. We will return to this example later.

\bigskip

\textbf{Example 1:} Let $S^{2}\subset\mathbb{R}^{3}$ be the sphere and $g$
the metric induced by the euclidean metric of $\mathbb{R}^{3}$. The
Levi-Civita connection $\Xi$ of $g$ is a connection on the orthonormal
oriented frame bundle $P\rightarrow S^{2}$, that has structure group $%
SO(2)\simeq U(1)$. The action of $G=SO(3)$ on $S^{2}$ lifts in a natural way
to an action on $P$. By Proposition \ref{PropHol} a) it is enough to study
the case of a rotation $\phi_{\alpha}$ of angle $\alpha$ around the $z$
axis. We can compute the $\phi_{\alpha}$-equivariant holonomy in geometrical
terms. The curvature of $\Xi$ coincides with the Euler form of $g$ and hence
we have $\mathrm{curv}(\Xi)=\frac{1}{2\pi}\mathrm{vol}_{g}$. If $\gamma$ is
a loop on $S^{2}$, by Proposition \ref{HolonomyInt} we have $\mathrm{Hol}%
^{\Xi}(\gamma)=\exp (i\cdot\mathrm{Area}(D))$, where $D\subset S^{2}$
satisfies $\gamma=\partial D$ (it exists because $S^{2}$ is simply
connected).

If $x$ is a point in $S^{2}$ and $\alpha\in\mathbb{R}/\mathbb{Z}$, we define 
$\sigma_{x,\alpha}\in\mathcal{C}_{x}^{\phi_{\alpha}}(M)$ by $%
\sigma_{x,\alpha }(t)=\phi_{t\alpha}(x)$. If $x$ is a point in the equator,
then $\sigma _{x,\alpha}$ is a geodesic. For any $y\in p^{-1}(x)$\ we have $%
\overline {\sigma_{x,\alpha}}^{y}(1)=(\phi_{\alpha})_{P}(y)$, and hence $%
\mathrm{Hol}_{\phi_{\alpha}}^{\Xi}(\sigma_{x,\alpha})=1$. If $\gamma\in%
\mathcal{C}_{x}^{\phi_{\alpha}}(M)$ then $\gamma\ast\overleftarrow{%
\sigma_{x,\alpha}}$ is a loop on $S^{2}$, and by Proposition \ref%
{HolonomyInt} d) we have%
\begin{equation*}
\mathrm{Hol}_{\phi_{\alpha}}^{\Xi}(\gamma)=\mathrm{Hol}^{\Xi}(\gamma \ast%
\overleftarrow{\sigma_{x,\alpha}})\cdot\mathrm{Hol}_{\phi_{\alpha}}^{\Xi
}(\sigma_{x,\alpha})=\exp(i\cdot\mathrm{Area}(D)),
\end{equation*}
where $\gamma\ast\overleftarrow{\sigma_{x,\alpha}}=\partial D$.

Finally, let $x$ be any point in $S^{2}$ and $\gamma\in\mathcal{C}%
_{x}^{\phi_{\alpha}}(M)$. We\ chose a curve $\zeta$ on $S^{2}$ joining $x$
and a point $x^{\prime}$ in the equator. Then by Proposition \ref%
{HolonomyInt} e)\ we have 
\begin{equation*}
\mathrm{Hol}_{\phi_{\alpha}}^{\Xi}(\gamma)=\mathrm{Hol}_{\phi_{\alpha}}^{\Xi
}(\overleftarrow{\zeta}\ast\gamma\ast(\phi_{\alpha}\cdot\zeta))=\exp (i\cdot%
\mathrm{Area}(D)),
\end{equation*}
where $\overleftarrow{\zeta}\ast\gamma\ast(\phi_{\alpha}\cdot\zeta )\ast%
\overleftarrow{\sigma_{x^{\prime},\alpha}}=\partial D$.

In particular, if $x=(0,0,1)$ then for any $\alpha\in\mathbb{R}/\mathbb{Z}$
we have $\phi_{\alpha}\in G_{x}$ and $\chi_{x}^{\Xi}\colon G_{x}\rightarrow
U(1)$ is given by $\chi_{x}^{\Xi}(\phi_{\alpha})=\exp(-i\alpha)$.$%
\blacksquare$

\section{Equivariant Curvature}

First, we recall the definition of equivariant cohomology in the Cartan
model (\emph{e.g. }see \cite{GS}). Suppose that we have a left action of a
connected Lie group $G$ on a manifold $M$. We denote by $\Omega^{k}(M)^{G}$
the space of $G$-invariant $k$-forms on $M$. Let $\Omega_{G}^{\bullet}(M)=%
\left( \mathbf{S}^{\bullet}(\mathfrak{g}^{\ast})\otimes\Omega^{\bullet}(M)%
\right) ^{G}$ be the space of $G$-invariant polynomials on $\mathfrak{g}$
with values in $\Omega^{\bullet}(M)$, with the graduation $\deg(\alpha)=2k+r$
if $\alpha\in\mathbf{S}^{k}(\mathfrak{g}^{\ast})\otimes\Omega^{r}(M)$. Let $%
D\colon\Omega_{G}^{q}(M)\rightarrow\Omega_{G}^{q+1}(M)$ be the Cartan
differential, $(D\alpha)(X)=d(\alpha(X))-\iota_{X_{M}}\alpha(X)$ for $X\in%
\mathfrak{g}$. On $\Omega_{G}^{\bullet}(M)$ we have $D^{2}=0$, and the
equivariant cohomology (in the Cartan model) of $M$ with respect of the
action of $G$ is defined as the cohomology of this complex.

Let $\varpi\in\Omega_{G}^{2}(M)$ be a $G$-equivariant $2$-form. Then we have 
$\varpi=\omega+\mu$ where $\omega\in\Omega^{2}(M)$ is $G$-invariant and $%
\mu\in\mathrm{Hom}\left( \mathfrak{g},\Omega^{0}(M)\right) ^{G}$. We have $%
D\omega=0\;$if\ and only if $d\omega=0$, and $\iota_{X_{M}}\omega=d(\mu
_{X})\;$for every $X\in\mathfrak{g}$. Hence $\mu$ is a comoment map for $%
\omega.$

If $\Xi$ is a $G$-invariant connection on a principal $U(1)$ bundle $%
P\rightarrow M$ then $\frac{i}{2\pi}D(\Xi)$ projects onto a closed $G$%
-equivariant $2$-form $\mathrm{curv}_{G}(\Xi)\in\Omega_{G}^{2}(M))$ called
the $G$-equivariant curvature of $\Xi$. If $X\in\mathfrak{g}$ then we have $%
\mathrm{curv}_{G}(\Xi)(X)=\mathrm{curv}(\Xi)+\mu_{X}^{\Xi}$, where $\mu
_{X}^{\Xi}=-\frac{i}{2\pi}\Xi(X_{P})$ is called the momentum of $\Xi$. If $%
\Xi^{\prime}$ is another $G$-invariant connection we have $%
\Xi^{\prime}=\Xi-2\pi i(p^{\ast}\lambda)$ for a $G$-invariant $%
\lambda\in\Omega^{1}(M)^{G}$. Then $\mathrm{curv}_{G}(\Xi^{\prime})=\mathrm{%
curv}_{G}(\Xi)+D\lambda$ and hence the equivariant cohomology class $[%
\mathrm{curv}_{G}(\Xi)]\in H_{G}^{2}(M)$ does not depend on the $G$%
-invariant connection chosen.

The Maurer-Cartan form on $U(1)$ is denoted by $\vartheta=z^{-1}dz$, and $%
\xi\in\mathfrak{X}(U(1))$ is the $U(1)$-invariant vector field $\xi (z)=iz$\
such that $\vartheta(\xi)=i$. We denote by $\xi_{P}\in\mathfrak{X}(P)$ the
vector field on $P$ corresponding to $\xi$.

For a topologically trivial bundle $M\times U(1)\rightarrow M$ the action of 
$G$ on $M\times U(1)$ is determined by a map $\theta\colon G\times
M\rightarrow U(1)$ characterized by the property $\phi_{P}(x,u)=(\phi
(x),u\cdot\theta_{\phi}(x))$. It satisfies the cocycle condition $\theta
_{\phi^{\prime}\cdot\phi}(x)=\theta_{\phi}(x)\cdot\theta_{\phi^{\prime}}(%
\phi(x))$. Conversely, any cocycle determines an action of $G$ on $M\times
U(1)$ such that it is\ a $G$-equivariant $U(1)$-bundle. In this case the
equivariant holonomy can be studied in terms of the cocycle $%
\theta_{\phi}(x) $ (e.g. see \cite{AnomaliesG}). For an arbitrary bundle, $%
\theta_{\phi}$ is defined only in a local trivialization. If $\Psi\colon%
\mathbb{R}^{n}\times U(1)\rightarrow P$ is a local trivialization covering a
map $\underline{\Psi }\colon\mathbb{R}^{n}\rightarrow M$ we have 
\begin{equation}
\Psi^{\ast}\Xi=\vartheta-2\pi i\underline{\Psi}^{\ast}\rho^{\Psi}
\label{trivialization}
\end{equation}
for a form $\rho^{\Psi}\in\Omega^{1}(\func{Im}\underline{\Psi})$. Note that
we have $\mathrm{curv}(\Xi)=d\rho^{\Psi}$ on $\func{Im}\underline{\Psi}$.

We denote by $\alpha\colon G\times M\rightarrow M$ the map defining the
action of $G$ on $M$, i.e., $\alpha(\phi,x)=\phi_{M}(x)$. Given $x\in M,$ we
can find neighborhoods $\mathcal{U}$ and $\mathcal{V}$ of $x$ and $e$ such
that $\mathcal{V}\times\mathcal{U}\subset\alpha^{-1}(\func{Im}\underline{\Psi%
})$. We define $\theta^{\Psi}\colon\mathcal{V}\times \mathcal{U}\rightarrow
U(1)$ by the property $\mathrm{pr}_{2}(\Psi^{-1}(\phi_{P}(y)))=\mathrm{pr}%
_{2}(\Psi^{-1}(y))\cdot\theta_{\phi}^{\Psi}(x)$ for any $y\in p^{-1}(x)$
with $x\in\mathcal{U}$ and $\phi\in\mathcal{V}$, and where $\mathrm{pr}%
_{2}\colon\mathbb{R}^{n}\times U(1)\rightarrow U(1)$ denotes the projection.
At the infinitesimal level, for any $X\in\mathfrak{g}$ we have 
\begin{equation}
\Psi_{\ast}^{-1}(X_{P})=\underline{\Psi}_{\ast}^{-1}(X_{M})+2\pi \mathfrak{a}%
_{X}^{\Psi}\cdot\xi,  \label{A(X)}
\end{equation}
where $\mathfrak{a}_{X}^{\Psi}(x)=\tfrac{i}{2\pi}\left. \tfrac{d}{dt}%
\right\vert _{t=0}\theta_{\exp(tX)}^{\Psi}(x)$. Using equations (\ref%
{trivialization}) and (\ref{A(X)}) we obtain the following

\begin{lemma}
\label{horizontalLift}a) The horizontal lift of $\gamma\colon I\rightarrow 
\func{Im}\Psi\subset M$ is given by $\overline{\gamma}^{\Psi
(z(0),u)}(s)=\Psi(z(s),u\cdot\exp(2\pi i{\textstyle\int\nolimits_{0}^{s}}%
\rho_{\gamma(t)}^{\Psi}(\dot{\gamma}(t))dt))$, where $z(s)=\underline{\Psi }%
^{-1}(\gamma(s))$.

b) $\mu_{X}^{\Xi}=\mathfrak{a}_{X}^{\Psi}-\rho^{\Psi}(X_{M})$.
\end{lemma}

If $\gamma\in\mathcal{C}_{x}^{\phi}(\func{Im}\underline{\Psi})$ with $x\in%
\mathcal{U}$ and $\phi\in\mathcal{V}$, then using Lemma \ref{horizontalLift}
a) we obtain

\begin{equation}
\mathrm{Hol}_{\phi}^{\Xi}(\gamma)=\theta_{\phi}^{\Psi}(x)^{-1}\cdot\exp%
\left( 2\pi i\tint _{\gamma}\rho^{\Psi}\right) .  \label{HolLocal}
\end{equation}

By Proposition \ref{HolonomyInt} the curvature of the connection\ measures
the infinitesimal holonomy. In a similar way the second term of the
equivariant curvature, the momentum $\mu^{\Xi}$, measures the infinitesimal
variation of the equivariant holonomy with respect to $\phi\in G$.

\begin{proposition}
\label{derivadaHol}Let $\varphi\colon(t_{0}-\varepsilon,t_{0}+\varepsilon
)\rightarrow G$ be a curve on $G$ with $\varphi_{t_{0}}=e$ and $x\in M$. If $%
X=\dot{\varphi}_{t_{0}}\in\mathfrak{g},$ and $\sigma_{x,t}(s)=(\varphi
_{t_{0}+s(t-t_{0})})_{M}(x)$, then $\sigma_{x,t}\in\mathcal{C}%
_{x}^{\varphi_{t}}(M)$\ and we have $\left. \tfrac{d}{dt}\right\vert
_{t=t_{0}}\mathrm{Hol}_{\varphi_{t}}^{\Xi}(\sigma_{x,t})=2\pi
i\mu_{X}^{\Xi}(x)$.
\end{proposition}

\begin{proof}
We chose a local trivialization $\Psi\colon\mathbb{R}^{n}\times
U(1)\rightarrow P$ and neighborhoods $\mathcal{U}$ and $\mathcal{V}$ of $x$
and $e$ such that $\mathcal{V}\times\mathcal{U}\subset\alpha^{-1}(\func{Im}%
\underline{\Psi})$. For $t$ close to $t_{0}$ we can use equation (\ref%
{HolLocal}) and we have 
\begin{equation*}
\mathrm{Hol}_{\varphi_{t}}^{\Xi}(\sigma_{x,t})=\theta_{\varphi_{t}}^{\Psi
}(x)^{-1}\cdot\exp(2\pi i\tint _{\sigma_{x,t}}\rho^{\Psi}).
\end{equation*}
Furthermore, we have 
\begin{align*}
\int_{\sigma_{x,t}}\rho^{\Psi} & =\int_{0}^{1}\rho^{\Psi}(\dot{\sigma}%
_{x,t}(s))ds=\int_{0}^{1}\rho^{\Psi}((\dot{\varphi}%
_{t_{0}+s(t-t_{0})})_{M}(x))(t-t_{0})ds \\
& \underset{r=t_{0}+s(t-t_{0})}{=}\int_{t_{0}}^{t}\rho^{\Psi}((\dot{\varphi }%
_{r}(x))_{M})dr.
\end{align*}
By taking the derivative and using equation (\ref{HolLocal}) we obtain 
\begin{align*}
\left. \tfrac{d}{dt}\right\vert _{t=t_{0}}\mathrm{Hol}_{\varphi_{t}}^{\Xi
}(\sigma_{x,t}) & =\left. \tfrac{d}{dt}\right\vert _{t=t_{0}}\theta
_{\varphi_{t}}^{\Psi}(x)^{-1}+\left. \tfrac{d}{dt}\right\vert
_{t=t_{0}}\exp(2\pi i\tint _{\sigma_{x,t}}\rho^{\Psi}) \\
& =2\pi i\mathfrak{a}_{X}^{\Psi}(x)-2\pi i\rho^{\Psi}(X_{M}(x))=2\pi i\mu
_{X}^{\Xi}(x).
\end{align*}
\end{proof}

\begin{proposition}
\label{InfHolonomy}For any $X\in\mathfrak{g}$ and $x\in M$ we define $%
\tau_{x,X}(s)=\exp(sX)_{M}(x)$. Then $\tau_{x,X}\in\mathcal{C}_{x}^{\exp
(X)}(M)$\ and we have $\mathrm{Hol}_{\exp(X)}^{\Xi}(\tau_{x,X})=\exp(2\pi
i\mu_{X}^{\Xi}(x))$.
\end{proposition}

\begin{proof}
We define $\sigma_{t}(s)=\exp(stX)_{M}\cdot x$ and we have $\sigma_{t}\in%
\mathcal{C}_{x}^{\exp(tX)}(M)$.\ By Proposition \ref{derivadaHol} we have $%
\left. \frac{d}{dt}\right\vert _{t=0}\mathrm{Hol}_{\exp(tX)}^{\Xi}(%
\sigma_{t})=2\pi i\mu_{X}^{\Xi}(x)$.

The curves $\sigma _{t+s}$ and $\sigma _{t}\ast (\exp (tX)\cdot \sigma _{s})$
differ by a reparametrization, and by Proposition \ref{PropHol} a) and b) we
conclude that $\mathrm{Hol}_{\exp ((t+s)X)}^{\Xi }(\sigma _{t+s})=\mathrm{Hol%
}_{\exp (tX)}^{\Xi }(\sigma _{t})\cdot \mathrm{Hol}_{\exp (sX)}^{\Xi
}(\sigma _{s})$. By taking the derivative we obtain%
\begin{align*}
\tfrac{d}{dt}\mathrm{Hol}_{\exp (tX)}^{\Xi }(\sigma _{t})& =\left. \tfrac{d}{%
ds}\right\vert _{s=0}\mathrm{Hol}_{\exp ((t+s)X)}^{\Xi }(\sigma _{t+s}) \\
& =\mathrm{Hol}_{\exp (tX)}^{\Xi }(\sigma _{t})\cdot \left. \tfrac{d}{ds}%
\right\vert _{s=0}\mathrm{Hol}_{\exp (sX)}^{\Xi }(\sigma _{s}) \\
& =2\pi i\mu _{X}^{\Xi }(x)\cdot \mathrm{Hol}_{\exp (tX)}^{\Xi }(\sigma _{t})%
\text{.}
\end{align*}%
The result follows by solving the differential equation and using that $\tau
_{x,X}=\sigma _{1}$ and that $\mathrm{Hol}_{\exp (0\cdot X)}^{\Xi }(\sigma
_{0})=1$ .
\end{proof}

\bigskip

\textbf{Example 1 (continuation):} We consider again example 1. We recall
that the symplectic form $\mathrm{vol}_{g}\in\Omega^{2}(S^{2})^{SO(3)}$
admits a canonical comoment map. For example it can be obtained by
considering $S^{2}$ as a coadjoint orbit of $SO(3)$ (e.g. see \cite[\S 4.6.1]%
{Abraham}). The map $\vec{v}\colon\mathfrak{so}(3)\longrightarrow\mathbb{R}%
^{3}$ that assigns to $X=\left( 
\begin{array}{rrr}
0 & a & b \\ 
-a & 0 & c \\ 
-b & -c & 0%
\end{array}
\right) \in\mathfrak{so}(3)$ the vector $\vec{v}_{X}=(c,-b,a)$ determines a
Lie algebra isomorphism between $(\mathfrak{so}(3),[,])$ and $(\mathbb{R}%
^{3},\times)$. Note that $X\in\mathfrak{so}(3)$ is an infinitesimal
generator of rotations around the axis determined by $\vec{v}_{X}$. We
define $h\colon\mathfrak{so}(3)\rightarrow\Omega^{0}(S^{2})$ by $%
h_{X}(x)=\langle \vec{v}_{X},x\rangle$ for $x\in S^{2}$ and $h$ is a
comoment map for $\mathrm{vol}_{g}$. Furthermore, the difference of two
comoment maps for $\mathrm{vol}_{g}$ determines an element of $H^{1}(%
\mathfrak{so}(3))=0$, and hence the comoment map is unique. We conclude from
this result that we have $\mu^{\Xi}=\frac{1}{2\pi}h$ and $\mathrm{curv}%
_{SO(3)}^{\Xi}=\frac{1}{2\pi }\mathrm{vol}_{g}+\frac{1}{2\pi}h$.

Let $X\in\mathfrak{so}(3)$ be the infinitesimal generator of rotations $%
\phi_{\alpha}$ around the z-axis, i.e., 
\begin{equation}
X=\left( 
\begin{array}{rrr}
0 & -1 & 0 \\ 
1 & 0 & 0 \\ 
0 & 0 & 0%
\end{array}
\right) .  \label{X generator}
\end{equation}
By Proposition \ref{InfHolonomy} we have $\mathrm{Hol}_{\phi_{\alpha}}^{\Xi
}(\tau_{x,\alpha X})=\exp(i\alpha h_{X}(x))$, a result that can be easily
seen to coincide with the previous result using the formula for the area of
a sector of the sphere.$\blacksquare$

\section{Equivariant holonomy and quotient}

In this section we assume that the actions of $G$ on $P$ and $M$ are free,
and that $P\overset{\bar{q}_{G}}{\rightarrow }P/G$ and $M\overset{q_{G}}{%
\rightarrow }M/G$\ are (left) principal $G$-bundles. This happens for
example if the action is free and proper (e.g. see \cite[page 264]{Abraham}%
). Then we have a well defined quotient bundle $P/G\rightarrow M/G$, and the
diagram%
\begin{equation*}
\begin{array}{ccc}
P & \overset{\bar{q}_{G}}{\rightarrow } & P/G \\ 
\downarrow &  & \downarrow \\ 
M & \overset{q_{G}}{\rightarrow } & M/G%
\end{array}%
\end{equation*}

Let $\Xi $ be a $G$-invariant connection on $P\rightarrow M$. Then $\Xi $ is 
$\bar{q}_{G}$-projectable onto a connection $\underline{\Xi }$ on $%
P/G\rightarrow M/G$ if and only if it is $G$-basic, and this is equivalent
to $\mu ^{\Xi }=0$. The next proposition shows that in this case the
holonomy of the connection $\underline{\Xi }$ can be computed in terms of
the $G$-equivariant holonomy of $\Xi $.

\begin{proposition}
If $\mu^{\Xi}=0$ then $\Xi$ projects onto a connection $\underline{\Xi}%
\in\Omega^{1}(P/G,i\mathbb{R})$. If $\gamma\in\mathcal{C}^{\phi}(M)$ then $%
\underline{\gamma}=q_{G}\circ\gamma$ is a loop on $M/G$\ and we have $%
\mathrm{Hol}_{\phi}^{\Xi}(\gamma)=\mathrm{Hol}^{\underline{\Xi}}(\underline{%
\gamma})$.
\end{proposition}

\begin{proof}
It follows from the fact that for any $y\in p^{-1}(\gamma(0))$ the curve $%
\bar{q}_{G}(\overline{\gamma}^{y})$ is a $\underline{\Xi}$-horizontal lift
of $\underline{\gamma}$.
\end{proof}

We note that the case commented in the Introduction of a proper action of a
discrete group $G$, is a particular case of the preceding proposition.

In the case in which $\mu^{\Xi}\neq0$, the connection $\Xi$\ is not
projectable to the quotient bundle. However, it is possible to obtain a
connection on the quotient using a connection $\Theta$ on the bundle $M%
\overset{q_{G}}{\rightarrow}M/G$. In more detail, we have the following
result (see \cite{flat})

\begin{proposition}
\label{quotcha} If $\Xi $ is a $G$-invariant connection on $p\colon
P\rightarrow M$ and $\Theta $ a connection on the (left) $G$-principal
bundle $q_{G}\colon M\rightarrow M/G$, then we define the $i\mathbb{R}$%
-valued $1$-form $\Xi (\Theta )(\xi )=\Xi ((\Theta (p_{\ast }\xi ))_{P})$, $%
\xi \in TP$. Then $\Xi -\Xi (\Theta )$ is projectable to $P/G$ and the
projection is a connection form $\underline{\Xi }_{\Theta }$ on $%
P/G\rightarrow M/G$.
\end{proposition}

The following result computes the holonomy of $\underline{\Xi}_{\Theta}$ in
terms of the $G$-equivariant holonomy of $\Xi$ and the holonomy of $\Theta$.

\begin{proposition}
\label{quotient}If $\gamma $ is a loop on $M/G$ and $\hat{\gamma}$ is a $%
\Theta $-horizontal lift of $\gamma $, then $\hat{\gamma}\in \mathcal{C}^{%
\mathrm{Hol}^{\Theta ,\hat{\gamma}(0)}(\gamma )}(M)$ and $\mathrm{Hol}^{%
\underline{\Xi }_{\Theta }}(\gamma )=\mathrm{Hol}_{\mathrm{Hol}^{\Theta ,%
\hat{\gamma}(0)}(\gamma )}^{\Xi }(\hat{\gamma})$.
\end{proposition}

\begin{proof}
By the definition of the holonomy of $\Theta $ we have $\hat{\gamma}(1)=\hat{%
\gamma}(0)\cdot \mathrm{Hol}^{\Theta ,\hat{\gamma}(0)}(\gamma )$, and hence $%
\hat{\gamma}\in \mathcal{C}^{\mathrm{Hol}^{\Theta ,\hat{\gamma}(0)}(\gamma
)}(M)$. For any $y\in p^{-1}(\hat{\gamma}(0))$ we have $\overline{\hat{\gamma%
}}^{y}(1)=\overline{\hat{\gamma}}^{y}(0)\cdot \mathrm{Hol}_{\mathrm{Hol}%
^{\Theta ,\hat{\gamma}(0)}(\gamma )}^{\Xi }(\hat{\gamma}).$

The curve $\bar{q}_{G}\circ \overline{\hat{\gamma}}^{y}$ is a $\underline{%
\Xi }_{\Theta }$-horizontal lift of $\gamma $ and $\bar{q}_{G}\circ 
\overline{\hat{\gamma}}^{y}(1)=\left( \bar{q}_{G}\circ \overline{\hat{\gamma}%
}^{y}(0)\right) \cdot \mathrm{Hol}_{\mathrm{Hol}^{\Theta ,\hat{\gamma}%
(0)}(\gamma )}^{\Xi }(\hat{\gamma})$. Hence $\mathrm{Hol}^{\underline{\Xi }%
_{\Theta }}(\gamma )=\mathrm{Hol}_{\mathrm{Hol}^{\Theta ,\hat{\gamma}%
(0)}(\gamma )}^{\Xi }(\hat{\gamma})$.
\end{proof}

\section{$G$-Flat connections\label{flat}}

A $G$-equivariant connection $\Xi$ is $G$-flat if $\mathrm{curv}_{G}(\Xi)=0$%
, i.e., if $\mathrm{curv}(\Xi)=0$ and $\mu^{\Xi}=0$. As commented in Section %
\ref{holonomy}, for flat connections the holonomy determines a homomorphism $%
\pi_{1}(M)\rightarrow U(1)$. For $G$-flat connections we have a similar
result, but with the $G$-equivariant fundamental group in place of $\pi
_{1}(M)$.

We define $\mathcal{C}_{x}^{G}(M)=\tbigcup \nolimits_{\phi\in G}\mathcal{C}%
_{x}^{\phi}(M)$. The $G$-equivariant fundamental group can be defined by $%
\pi_{1,G}(M)_{x}=\mathcal{C}_{x}^{G}(M)/\sim_{G}$\ where $%
(\phi,\gamma)\sim_{G}(\phi^{\prime},\gamma^{\prime})$ if there exist a curve 
$\varphi\colon I\rightarrow G$ , and a continuous map $h\colon I\times
I\rightarrow\mathbb{R}$, $(t,s)\mapsto$ $h_{t}(s)$ such that $h_{t}\in%
\mathcal{C}_{x}^{\varphi_{t}}(M)$ for any $t\in I$ and $\varphi_{0}=\phi$, $%
\varphi_{1}=\phi^{\prime}$, $h_{0}=\gamma$, $h_{1}=\gamma^{\prime}$. We
define the product in $\pi_{1,G}(M)_{x}$ by $[(\phi,\gamma)]\ast\lbrack
(\phi^{\prime},\gamma^{\prime})]=[(\phi\cdot\phi^{\prime},\gamma\ast(\phi
\cdot\gamma^{\prime})]$. As usual, the groups $\pi_{1,G}(M)_{x}$ for
different points $x$ are isomorphic, and hence we suppress the point $x$ in
the notation.

\begin{lemma}
\label{LemmaOrbita}If $\mu ^{\Xi }=0$ then for any curve $\varphi \colon
I\rightarrow G$ and $x\in M$ we have $\mathrm{Hol}_{\phi }^{\Xi }(\gamma )=1$%
, where $\phi =\varphi _{1}\cdot \varphi _{0}^{-1}$ and $\gamma (s)=(\varphi
_{s})_{M}(x)$.
\end{lemma}

\begin{proof}
We define $\gamma _{t}(s)=\gamma (st)$ and $k(t)=\mathrm{Hol}_{\varphi
_{t}\cdot \varphi _{0}^{-1}}^{\Xi }(\gamma _{t})$. The result follows if we
prove that $\tfrac{dk}{dt}=0$. We fix $t_{0}\in I$ and we define $\gamma
_{t_{0},t}(s)=\gamma (t_{0}+s(t-t_{0}))$. The curves $\gamma _{t}$ and $%
\gamma _{t_{0}}\ast \gamma _{t_{0},t}$ differ by a reparametrization and
hence $k(t)=\mathrm{Hol}_{\varphi _{t}\cdot \varphi _{0}^{-1}}^{\Xi }(\gamma
_{t_{0}}\ast \gamma _{t_{0},t})=\mathrm{Hol}_{\varphi _{t_{0}}\cdot \varphi
_{0}^{-1}}^{\Xi }(\gamma _{t_{0}})\cdot \mathrm{Hol}_{\varphi _{t}\cdot
\varphi _{t_{0}}^{-1}}^{\Xi }(\gamma _{t_{0},t})$. Furthermore we have $%
\tfrac{dk}{dt}(t_{0})=\left. \tfrac{d}{dt}\right\vert _{t=0}k(t_{0}+t)=%
\mathrm{Hol}_{\varphi _{t_{0}}\cdot \varphi _{0}^{-1}}^{\Xi }(\gamma
_{t_{0}})\cdot \left. \tfrac{d}{dt}\right\vert _{t=t_{0}}\mathrm{Hol}%
_{\varphi _{t}\cdot \varphi _{t_{0}}^{-1}}^{\Xi }(\gamma _{t_{0},t})$.

Using Proposition \ref{PropHol} a) we obtain$\ \mathrm{Hol}%
_{\varphi_{t}\cdot\varphi_{t_{0}}^{-1}}^{\Xi}(\gamma_{t_{0},t})=\mathrm{Hol}%
_{\varphi
_{t_{0}}^{-1}\varphi_{t}}^{\Xi}(\varphi_{t_{0}}^{-1}\cdot\gamma_{t_{0},t})$.
The result follows by applying Proposition \ref{derivadaHol} to the curve $%
\varphi_{t_{0}}^{-1}\cdot\varphi$.
\end{proof}

\begin{proposition}
\label{HolGflat}Let $\Xi$ be a $G$-flat connection on $P\rightarrow M$. If $%
(\phi,\gamma)\sim_{G}(\phi^{\prime},\gamma^{\prime})$ then $\mathrm{Hol}%
_{\phi}^{\Xi}(\gamma)=\mathrm{Hol}_{\phi^{\prime}}^{\Xi}(\gamma^{\prime})$.
Hence the $G$-equivariant holonomy induces a group homomorphism $\mathrm{Hol}%
_{G}^{\Xi}\colon\pi_{1,G}(M)\rightarrow U(1)$.
\end{proposition}

\begin{proof}
Let $\varphi\colon I\rightarrow G$ and $h\colon I\times I\rightarrow \mathbb{%
R}$ ,$(t,s)\mapsto$ $h_{t}(s)$ such that $h_{t}\in\mathcal{C}%
_{x}^{\varphi_{t}}(M)$ for any $t\in I$ and $\varphi_{0}=\phi$, $\varphi
_{1}=\phi^{\prime}$, $h_{0}=\gamma$, $h_{1}=\gamma^{\prime}$. We define $%
\sigma(s)=(\varphi_{s})_{M}(x)$ and we have $\sigma\in\mathcal{C}%
^{\phi^{\prime}\cdot\phi^{-1}}(M)$ and $\partial h=\gamma\ast\sigma \ast%
\overleftarrow{\gamma}^{\prime}$.As $\Xi$ is flat we have $\mathrm{Hol}%
_{e}^{\Xi}(\gamma\ast\sigma\ast\overleftarrow{\gamma}^{\prime})=\exp(2\pi
i\int_{I\times I}h^{\ast}\mathrm{curv}(\Xi))=1$. But by Proposition \ref%
{PropHol} we also have $\mathrm{Hol}_{e}^{\Xi}(\gamma\ast\sigma \ast%
\overleftarrow{\gamma}^{\prime})=\mathrm{Hol}_{\phi}^{\Xi}(\gamma )\cdot%
\mathrm{Hol}_{\phi^{\prime}\cdot\phi^{-1}}^{\Xi}(\sigma)\cdot \mathrm{Hol}%
_{\phi^{\prime}}^{\Xi}(\gamma^{\prime})^{-1}$. The result follows because $%
\mathrm{Hol}_{\phi^{\prime}\cdot\phi^{-1}}^{\Xi}(\sigma)=1$ by Lemma \ref%
{LemmaOrbita}.
\end{proof}

The projection $(\phi,\gamma)\mapsto\phi$ induces an epimorphisms $\pi
_{1,G}(M)\longrightarrow\pi_{0}(G)$. In a similar way, the inclusion map $%
\mathcal{C}_{x}(M)\rightarrow\mathcal{C}_{x}^{G}(M)$ indices a homomorphism%
\footnote{%
We recall that it is equivalent to define the fundamental group $\pi_{1}(M)$
by using continous or differentiable curves (e.g. see \cite[17.8.1]{BT}).} $%
\pi_{1}(M)\rightarrow\pi_{1,G}(M)$. It can be seen that we have an exact
sequence 
\begin{equation}
\pi_{1}(M)\rightarrow\pi_{1,G}(M)\rightarrow\pi_{0}(G)\rightarrow1.
\label{exact}
\end{equation}

\begin{remark}
It is possible to give another interpretation of this result in terms of the
Borel model of equivariant cohomology\ \ If $EG\rightarrow BG$ is a
universal bundle for $G$ then we define the homotopy quotient $%
M_{G}=(M\times EG)/G$.\ Then $\mathrm{pr}_{1}^{\ast }\Xi $ is a $G$-flat
connection on $P\times EG\rightarrow M\times EG$ and it projects onto a flat
connection $\Xi _{G}$ on $P_{G}\rightarrow M_{G}$. As $\Xi _{G}$ is flat,
its holonomy defines a homomorphism $\pi _{1}(M_{G})\rightarrow U(1)$ that
corresponds to the one defined in Proposition \ref{HolGflat}. It can be seen
that we have isomorphisms $\pi _{1,G}(M)\simeq \pi _{1,G}(M\times EG)\simeq
\pi _{1}(M_{G})$. Furthermore, the homotopy exact sequence induces the
following exact sequence $\pi _{1}(G)\rightarrow \pi _{1}(M)\rightarrow \pi
_{1}(M_{G})\rightarrow \pi _{0}(G)\rightarrow 1$. We prefer to work with $%
\pi _{1,G}(M)$ in place of $\pi _{1}(M_{G})$ because it is defined in terms
of curves on $M$ and $G$, and hence it can be related directly with the
equivariant holonomy.
\end{remark}

\section{Classification of equivariant $U(1)$-bundles by their equivariant
holonomy}

In this section we obtain the equivariant versions of the results of Section %
\ref{holonomy}. If $p\colon P\rightarrow M$, $p^{\prime }\colon P^{\prime
}\rightarrow M$\ are\ two $G$-equivariant $U(1)$-bundles then we write that $%
P\simeq _{G}P^{\prime }$ if there exists a $G$-equivariant $U(1)$-bundle
isomorphism $\Phi \colon P^{\prime }\rightarrow P$ covering the identity map
of $M$. We say that $P$ is a trivial $G$-equivariant $U(1)$-bundle if $%
P\simeq _{G}M\times U(1)$ for an action of $G$ on $M$ and where $G$ acts
trivially on $U(1)$. A $G$-equivariant $U(1)$-bundle with connection is a
pair $(P,\Xi )$, where $p\colon P\rightarrow M$ is a $G$-equivariant $U(1)$%
-bundle and $\Xi $ is a $G$-invariant connection on $P$. We write that $%
(P,\Xi )\simeq _{G}(P,\Xi )$ if there exists a $G$-equivariant $U(1)$-bundle
isomorphism $\Phi \colon P^{\prime }\rightarrow P$ covering the identity map
of $M$ such that \ $\Phi ^{\ast }\Xi =\Xi ^{\prime }$.

\begin{theorem}
\label{IsoEqui}If $(P,\Xi)$ and $(P^{\prime},\Xi^{\prime})$ are $G$%
-equivariant $U(1)$-bundles with connection over $M$, then $(P,\Xi)\simeq
_{G}(P^{\prime},\Xi^{\prime})$ if and only if $\mathrm{Hol}_{\phi}^{\Xi
}(\gamma)=\mathrm{Hol}_{\phi}^{\Xi^{\prime}}(\gamma)$ for any $\phi\in G$,
and $\gamma\in\mathcal{C}^{\phi}(M)$.
\end{theorem}

\begin{proof}
That equivalent bundles have the same equivariant holonomy follows from
Proposition \ref{PropHol} f). We prove the converse. If $\mathrm{Hol}_{\phi
}^{\Xi }(\gamma )=\mathrm{Hol}_{\phi }^{\Xi ^{\prime }}(\gamma )$ for any $%
\phi \in G$, and $\gamma \in \mathcal{C}^{\phi }(M)$, in particular we have $%
\mathrm{Hol}^{\Xi }(\gamma )=\mathrm{Hol}^{\Xi ^{\prime }}(\gamma )$ for any 
$\gamma \in \mathcal{C}^{e}(M)=\mathcal{C}(M)$ and by Theorem \ref{Clas}
there exists a a $U(1)$-bundle isomorphism $\Phi \colon P\rightarrow
P^{\prime }$ covering the identity map of $M$ such that $\Phi ^{\ast }(\Xi
^{\prime })=\Xi $. We prove that $\Phi $ is $G$-equivariant. If $\overline{%
\gamma }^{y}$ is the $\Xi $-horizontal lift of $\gamma \in \mathcal{C}^{\phi
}(M)$ starting at $y$, then $\Phi \circ \overline{\gamma }^{y}$ coincides
with the $\Xi ^{\prime }$-horizontal lift $\overline{\gamma }^{\prime \Phi
(y)}$\ of $\gamma $. By applying equation (\ref{def}) we obtain 
\begin{align*}
\Phi (\phi _{P}(y))& =\Phi (\overline{\gamma }^{y}(1)\cdot \mathrm{Hol}%
_{\phi }^{\Xi }(\gamma )^{-1})=\Phi (\overline{\gamma }^{y}(1))\cdot \mathrm{%
Hol}_{\phi }^{\Xi }(\gamma )^{-1} \\
& =\overline{\gamma }^{\prime \Phi (y)}(1)\cdot \mathrm{Hol}_{\phi }^{\Xi
^{\prime }}(\gamma )^{-1}=\phi _{P^{\prime }}(\Phi (y)).
\end{align*}
\end{proof}

\begin{proposition}
\label{PropChangeConn}If $\Xi$ and $\Xi^{\prime}$ are $G$-equivariant
connections on $P\rightarrow M$, then we have $\Xi^{\prime}=\Xi-2\pi
i(p^{\ast}\rho)$ for a $G$-invariant form $\rho\in\Omega^{1}(M)$ and $%
\mathrm{Hol}_{\phi}^{\Xi^{\prime}}(\gamma)=\mathrm{Hol}_{\phi}^{\Xi}(\gamma)%
\cdot\exp(2\pi i\int_{\gamma}\rho)$.
\end{proposition}

\begin{proof}
It follows from the fact that if $\overline{\gamma}$ is a $\Xi$-horizontal
lift of $\gamma$, then $\overline{\gamma}^{\prime}(s)=\overline{\gamma }%
(s)\cdot\exp(2\pi i\int_{0}^{s}\rho_{\gamma(s)}(\dot{\gamma}(s))ds)$ is a $%
\Xi^{\prime}$-horizontal lift of $\gamma$.
\end{proof}

\bigskip

The following result generalizes to arbitrary $U(1)$-bundles the result used
in \cite{AnomaliesG} to study anomaly cancellation

\begin{theorem}
\label{anomaly}Let $(P,\Xi)$ be a $G$-equivariant $U(1)$-bundle with
connection over $M$. Then $P\rightarrow M$ is a trivial $G$-equivariant $%
U(1) $-bundle if and only if there exists a $G$-invariant $1$-form $\beta
\in\Omega^{1}(M)^{G}$ such that $\mathrm{Hol}_{\phi}^{\Xi}(\gamma
)=\int_{\gamma}\beta$ for any $\phi\in G$ and $\gamma\in\mathcal{C}%
^{\phi}(M) $.
\end{theorem}

\begin{proof}
If $P\rightarrow M$ is a trivial $G$-equivariant $U(1)$-bundle, we can
choose a global trivialization $\Psi\colon M\times U(1)\longrightarrow P$
with $\theta_{\phi}^{\Psi}(x)=1$ and hence $\mathrm{Hol}_{\phi}^{\Xi}(\gamma
)=\exp(2\pi i\tint _{\gamma}\rho^{\Psi})$ and we can take $\beta=\rho^{\Psi}$%
.

Conversely, if $\mathrm{Hol}_{\phi }^{\Xi }(\gamma )=\exp (2\pi
i\tint_{\gamma }\beta )$, then we can define a new $G$-invariant connection $%
\Xi ^{\prime }=\Xi +2\pi i\beta $ and by Proposition \ref{PropChangeConn} we
have $\mathrm{Hol}_{\phi }^{\Xi ^{\prime }}(\gamma )=1=\mathrm{Hol}_{\phi
}^{\vartheta }(\gamma )$ for any $\phi \in G$ and $\gamma \in \mathcal{C}%
^{\phi }(M)$, and by Theorem \ref{IsoEqui} we have $(P,\Xi ^{\prime })\simeq
_{G}(M\times U(1),\vartheta )$. In particular $P\simeq _{G}M\times U(1)$.
\end{proof}

\section{Equivariant prequantization and equivariant holonomy}

A $D$-closed equivariant $2$-form $\varpi =\omega +\mu \in \Omega
_{G}^{2}(M) $ is $G$-equivariant prequantizable if there exist a $G$%
-equivariant $U(1)$-bundle with connection $(P,\Xi )$ such that $\mathrm{curv%
}_{G}(\Xi )=\varpi $. By Weil-Kostant theorem (e.g. see \cite{Kostant}) a
necessary condition is that $\omega $ should be integral, but it is known
that there could be additional obstructions (e.g. see \cite{Mundet}).
Necessary conditions for equivariant prequantizability follow from
Proposition \ref{InfHolonomy}. If $(P,\Xi )$ is a $G$-equivariant
prequantization of $\varpi $ and $X\in \mathfrak{g}$ satisfies $\exp (X)=e$,
then by Proposition \ref{InfHolonomy}\ we have $\exp (2\pi i\mu _{X}^{\Xi
}(x))=\mathrm{Hol}^{\Xi }(\tau _{x,X})$ for any $x\in M$, where $\tau
_{x,X}(t)=\exp (tX)_{M}(x)$ is the curve defined in Proposition \ref%
{InfHolonomy}. If this condition is not satisfied, then $\varpi $ is not $G$%
-equivariant prequantizable. We apply this condition to Example 1.

\textbf{Example 1 (continuation): }We consider again the case of $%
S^{2}\subset\mathbb{R}^{3}$ and the action of $G=SO(3)$. We have seen that $%
\varpi=\frac{1}{2\pi}\mathrm{vol}_{g}+\frac{1}{2\pi}h\in%
\Omega_{SO(3)}^{2}(S^{2})$ is $SO(3)$-equivariant prequantizable. The form $%
\frac{1}{4\pi }\mathrm{vol}_{g}$ is integral, and hence it is
prequantizable. However, we have the following result

\begin{proposition}
The form $\varpi^{\prime}=\frac{1}{4\pi}\mathrm{vol}_{g}+\frac{1}{4\pi}%
h\in\Omega_{SO(3)}^{2}(S^{2})$ is not $SO(3)$-equivariant prequantizable.
\end{proposition}

\begin{proof}
We consider the point $x=(0,0,1)$, and the vector $Y=-2\pi X\in\mathfrak{so}%
(3)$, where $X$ is the infinitesimal generator of rotations around the
z-axis of equation (\ref{X generator}). If $\varpi^{\prime}=\omega^{\prime}+%
\mu^{\prime}$, we have $\exp(Y)=e$, $\mu_{Y}^{\prime}(x)=\frac{1}{4\pi}%
h_{Y}(x)=\frac{1}{2}$ and the curve $\tau_{x,Y}$\ is a constant curve with
value $x$. Hence if $(P,\Xi)$ is any $U(1)$-bundle with connection such that 
$\mathrm{curv}(\Xi)=\omega^{\prime}$ we have $1=\mathrm{Hol}^{\Xi}(\tau
_{x,Y})\neq\exp(2\pi i\mu_{Y}^{\prime}(x))=\exp(\pi i)=-1$.
\end{proof}

\bigskip

\section{Anomalies and equivariant holonomy\label{anomalies}}

In this section we study the application of the results of the present paper
to the study of anomalies in quantum field theory. As commented in the
Introduction this was our original motivation to study equivariant holonomy.
Let $\mathcal{M}_{M}$ be the space of Riemannian metrics on $M$ and $%
\mathcal{D}_{M}$ the group of orientation preserving diffeomorphisms.
Anomalies appear when a classical symmetry of a theory is broken at the
quantum level. This happens for example in theories with chiral Dirac
operators, where the path integral $Z\in \Omega ^{0}(\mathcal{M}_{M})$
(defined as a regularized determinant) fails to be $\mathcal{D}_{M}$%
-invariant. In this case, for $g\in \mathcal{M}_{M}$ and $\phi \in \mathcal{D%
}_{M}$ we have $Z(\phi (g))=Z(g)\cdot \theta (\phi ,g)$ , where $\theta
\colon \mathcal{D}_{M}\times \mathcal{M}_{M}\rightarrow U(1)$ satisfies the
cocycle condition. Hence $\theta $ defines a $\mathcal{D}_{M}$-equivariant $%
U(1)$-bundle $P\rightarrow \mathcal{M}_{M}$. This bundle is called the
anomaly bundle and admits a $\mathcal{D}_{M}$-invariant connection $\Xi $
(e.g. see \cite{freed}). If the gravitational anomaly cancels, the anomaly
bundle admits a $\mathcal{D}_{M}$-invariant section (e.g. see \cite%
{AnomaliesG}) and $P\rightarrow \mathcal{M}_{M}$ is a trivial $\mathcal{D}%
_{M}$-equivariant $U(1)$-bundle. A topological\ obstruction for anomaly
cancellation can be obtained by considering the quotient bundle$\footnote{%
In order to have a well defined quotient manifold it is necesary to restrict
the group $\mathcal{D}_{M}$\ to a subgroup acting freely on $\mathcal{M}_{M}$%
, but we omit this point here. Furthermore, one of the advantages of working
with equivariant holonomy is that this restriction is unnecessary.}$ $P/%
\mathcal{D}_{M}\rightarrow \mathcal{M}_{M}/\mathcal{D}_{M}$. If this
quotient bundle is non-trivial, then the anomaly cannot be cancelled. Hence
the Chern class of $P/\mathcal{D}_{M}$ represents an obstruction for anomaly
cancellation. This allows us to interpret the gravitational anomaly as a
cohomology class on $\mathcal{M}_{M}/\mathcal{D}_{M}$. Furthermore, the
principal $\mathcal{D}_{M}$-bundle $\mathcal{M}_{M}\rightarrow \mathcal{M}%
_{M}/\mathcal{D}_{M}$ admits a natural connection $\Theta $. By Proposition %
\ref{quotcha} the connections $\Xi $ and $\Theta $ determine a connection $%
\underline{\Xi }_{\Theta }$ on $P/\mathcal{D}_{M}\rightarrow \mathcal{M}_{M}/%
\mathcal{D}_{M}$. The curvature of $\underline{\Xi }_{\Theta }$ can be
computed using the Atiyah-Singer theorem. Moreover, in \cite{WittenGlobAn}
Witten introduces a formula that measures the variation of the path integral 
$Z$ along a curve $\gamma \colon I\rightarrow \mathcal{M}_{M}$ such that $%
\gamma (1)=\phi (\gamma (0))$, i.e., he defines a map $w\colon \mathcal{C}%
^{\phi }(\mathcal{M}_{M})\rightarrow U(1)$. Later Witten's formula was
interpreted (e.g. see \cite{freed}) as a computation of the holonomy of the
connection $\underline{\Xi }_{\Theta }$. By Proposition \ref{quotient} $w$
can also be considered as a computation of the $\mathcal{D}_{M}$-equivariant
holonomy $\Xi $, and this result is applied in \cite{LocUni}. Furthermore,
Theorem \ref{anomaly} provides necessary and sufficient conditions for
anomaly cancellation.\ 

If the equivariant curvature of $\Xi $ vanishes (this happens for example if 
$\dim M\neq 2\func{mod}4$), then $\Xi $ is a $\mathcal{D}_{M}$-flat
connection and by the results of Section \ref{flat}\ the equivariant
holonomy of $\Xi $ defines a homomorphism $w\colon \pi _{1,G}(\mathcal{M}%
_{M})\rightarrow U(1)$. As $\mathcal{M}_{M}$ is simply connected we conclude
from the exact sequence (\ref{exact}) that $\pi _{1,G}(\mathcal{M}%
_{M})\simeq \pi _{0}(\mathcal{D}_{M})$ is the mapping class group of $M$.
Hence we obtain a homomorphism $w\colon \pi _{0}(\mathcal{D}_{M})\rightarrow
U(1)$, that in quantum field theory is called a global gravitational anomaly.

The preceding geometrical interpretation of anomalies is insufficient from
the physical point of view due to the problem of locality (e.g. see \cite%
{ASZ}, \cite{AnomaliesG}). In quantum field theory, the gravitational
anomalies can be cancelled only using local terms, i.e. terms obtained by
integration over $M$ of forms depending on the metric and its derivatives.
Geometrically this implies that to cancel gravitational anomalies the
existence of an especial type of $\mathcal{D}_{M}$-equivariant section of
the anomaly bundle $P\rightarrow $\ $\mathcal{M}_{M}$ is necessary (see \cite%
{AnomaliesG}). The principal problem in the study of locality in anomaly
cancellation is that the connection $\Theta $ contains non-local terms and
hence it is difficult to deal with the locality problem using the quotient
bundle $P/\mathcal{D}_{M}\rightarrow \mathcal{M}_{M}/\mathcal{D}_{M}$ and
the connection $\underline{\Xi }_{\Theta }$. However, the $\mathcal{D}_{M}$%
-equivariant curvature and holonomy of $\Xi $ have local expressions (see 
\cite{anomalies}, \cite{AnomaliesG}, \cite{LocUni}), and Theorem \ref%
{anomaly} can be extended to characterize gravitational anomaly cancellation
in a way compatible with locality (see \cite[Proposition 20]{LocUni}).

Finally we note that in the case of gravitational anomalies the space of
fields $\mathcal{M}_{M}$ is contractible and in this case the study of
equivariant holonomy can be simplified (e.g. see \cite{AnomaliesG}).
However, if we consider other theories (for example sigma models or string
theory), the space of fields can be not contractible. In those cases, the
study of anomaly cancellation requires the equivariant holonomy of
connections on arbitrary bundles that we consider in this paper.


\begin{thebibliography}{99}
\bibitem{Abraham} R. Abraham and J. E. Marsden, \emph{Foundations of
Mechanics}, Second edition, Addison-Wesley, Redwood City, Ca., 1987.

\bibitem{ASZ} O. \'{A}lvarez, I. Singer, B. Zumino, \emph{Gravitational
Anomalies and the Family's Index Theorem}, Commun. Math. Phys. \textbf{96}
(1984) 409--417.

\bibitem{EquiChern} D. Berwick-Evans, F. Han, \emph{The equivariant Chern
character as super holonomy on loop stacks}, arXiv:1610.02362.

\bibitem{BT} R. Bott, L. Tu, \emph{Equivariant characteristic classes in the
Cartan model}, Geometry, analysis and applications (Varanasi, 2000), 3--20,
World Sci. Publishing, River Edge, NJ, 2001.

\bibitem{Brylinsky} J.L. Brylinski,\emph{\ Loop Spaces, Characteristic
Classes and Geometric Quantization}, Progress in Math. vol 107 Birkh\"{a}%
user-Boston (1993).

\bibitem{BV2} N. Berline, M. Vergne, \emph{Z\'{e}ros d'un champ de vecteurs
et classes charact\'{e}ristiques \'{e}quivariantes}, Duke Math. J. \textbf{50%
} (1983) 539--549.

\bibitem{flat} M. Castrill\'{o}n L\'{o}pez, R. Ferreiro P\'{e}rez, \emph{%
Differential characters and cohomology of the moduli of flat connections},
Lett. Math. Phys \textbf{109} (2019) 11--31.

\bibitem{Cheeger} J. Cheeger and J. Simons, \emph{Differential characters
and geometric invariants}, in Geometry and Topology, Proc. Spec. Year,
College Park/Md. 1983/84, Lecture Notes in Math. 1167, Springer--Verlag,
Berlin--Heidelberg--New York, 1985.

\bibitem{anomalies} R. Ferreiro P\'{e}rez, \emph{Local anomalies and local
equivariant cohomology}, Comm. Math. Phys. \textbf{286} (2009) 445--458.

\bibitem{CSconnections} ---, \emph{Equivariant prequantization bundles on
the space of connections and characteristic classes}, Ann. Mat. Pura Appl. 
\textbf{197} (2018) 1749--1770.

\bibitem{AnomaliesG} ---, \emph{On the geometrical interpretation of
locality in anomaly cancellation}\textbf{,} J. Geom. Phys. 133 (2018)
102--112.

\bibitem{LocUni} ---, \emph{Locality and universality in gravitational
anomaly cancellation}, preprint arXiv:1805.12068.

\bibitem{freed} D.S. Freed, \emph{Determinants, torsion, and strings}, Comm.
Math. Phys. \textbf{107} (1986), no. 3, 483--513.

\bibitem{Gomi} K. Gomi. \emph{Equivariant smooth Deligne cohomology}, Osaka
J. Math

\textbf{42} (2005), 309--337.

\bibitem{GS} V.\ Guillemin, S.\ Sternberg, \emph{Supersymmetry and
Equivariant de Rham Theory}, Springer-Verlag, Berlin Heidelberg, 1999.

\bibitem{KN1} S.\ Kobayashi, K.\ Nomizu, \emph{Foundations of Differential
Geometry}, John Wiley \& Sons, Inc.\ (Interscience Division), New York,
Volume I, 1963; Volume II, 1969.

\bibitem{Kostant} B. Kostant, \emph{Quantization and unitary representations}%
, Lectures in modern analysis and applications III, Lecture Notes in Math.
170 (Springer, New York, 1970), 87--208.

\bibitem{kubel} A. K\"{u}bel, A. Thom, \emph{Equivariant differential
cohomology}, Trans. Amer. Math. Soc. \textbf{370} (2018), 8237--8283.

\bibitem{LermanMalkin} E. Lerman, A. Malkin,\emph{\ Equivariant Differential
Characters and Symplectic Reduction}\textbf{,} Commun. Math. Phys. \textbf{%
289}\ (2009) 777--801.

\bibitem{Mundet} I. Mundet i Riera, \emph{Lifts of Smooth Group Actions to
Line Bundles}, Bull. London Math. Soc., \textbf{33} (2001), 351--361.

\bibitem{WittenGlobAn} E. Witten, \emph{Global Gravitational Anomalies},
Comm. Math. Phys. \textbf{100} (1985), 197--229.
\end{thebibliography}
\end{document}